\documentclass[11pt, reqno]{amsart}
\usepackage{lineno,hyperref}

\usepackage{cases}
\usepackage[square, comma, sort&compress, numbers]{natbib}
\usepackage{float}
\usepackage{pifont}
\usepackage{bbding}
\usepackage{amssymb}
\usepackage{mathrsfs}
\usepackage{subfigure}
\usepackage{caption}
\usepackage{geometry}
\usepackage{color}
\usepackage{amsmath, amsthm, amscd, amsfonts, amssymb, graphicx, color}
\usepackage{lineno}
\newtheorem{theorem}{Theorem}
\newtheorem{lemma}{Lemma}

\newtheorem{definition}{Definition}

\newtheorem{remark}{Remark}
\newtheorem{example}{Example}

\numberwithin{equation}{section}

\newcommand{\abs}[1]{\left\vert#1\right\vert}
\newcommand{\U}{\mbox{$\mathbb{U}$}}

\newcommand{\N}{\mbox{$\mathbb{N}$}}

\date{}
\begin{document}
\setcounter{page}{1}

\title[On a subclass of starlike functions]
{On a subclass of starlike functions associated\\ with a vertical strip domain}

\author[Yong Sun, Zhi-Gang Wang, Antti Rasila and Janusz Sok\'{o}{\l}]{Yong Sun, Zhi-Gang Wang, Antti Rasila and Janusz Sok\'{o}{\l}}

\vskip.10in
\address{\noindent Yong Sun\vskip.05in
School of Science, Hunan Institute of Engineering,
Xiangtan, 411104, Hunan, People's Republic of China P. R. China.}
\vskip.05in
\email{\textcolor[rgb]{0.00,0.00,0.84}{yongsun2008$@$foxmail.com}}

\vskip.05in
\address{\noindent Zhi-Gang Wang\vskip.05in
 School of Mathematics and Computing Science, Hunan
First Normal University, Changsha 410205, Hunan, P. R. China.}
\vskip.05in
\email{\textcolor[rgb]{0.00,0.00,0.84}{wangmath$@$163.com}}

\address{\noindent Antti Rasila\vskip.05in
Guangdong Technion-Israel Institute of Technology, 241 Daxue Road, Shantou 515063, Guangdong,
P. R. China.}
\email{\textcolor[rgb]{0.00,0.00,0.84}{rasila$@$iki.fi}}

\address{\noindent Janusz Sok\'{o}{\l}\vskip.05in
Faculty of Mathematics and Natural Sciences, University of Rzesz\'{o}w,
ul. Rejtana 16A, 35-310 Rzesz\'{o}w, Poland.}
\email{\textcolor[rgb]{0.00,0.00,0.84}{jsokol$@$prz.edu.pl}}


\subjclass[2010]{Primary 30C45; Secondary 30C55.}

\keywords{Analytic function; univalent function; starlike function; differential subordination.}

\begin{abstract}
In this paper, we consider a subclass of starlike functions associated with a vertical strip domain.
Several results concerned with integral representations, convolutions, and
coefficient inequalities for functions belonging to this class are obtained.
Furthermore, we consider radius problems and inclusion relations
involving certain classes of strongly starlike functions, parabolic starlike functions and other types of starlike functions.
The results are essential improvements of the corresponding results obtained by Kargar {\it et al.}, and the derivations are similar to those used earlier by Sun
{\it et al.} and Kwon {\it et al.}.
\end{abstract}

\vskip.20in

\maketitle

\tableofcontents

\section{Introduction}

Let $\mathcal{A}$ denote the class of the functions of the form:
\begin{equation}\label{11}
f(z)=z+\sum_{n=2}^{\infty}a_{n}z^{n},
\end{equation}
which are analytic and univalent in the open unit disk $\U=\{z\in\mathbb{C}:\ |z|<1\}$. A function $f\in\mathcal{A}$ is
said to be {\it starlike of order} $\beta\ (0\leq \beta<1)$, if it satisfies the condition:
\begin{equation*}
\Re\bigg(\frac{zf'(z)}{f(z)}\bigg)>\beta \quad (z\in\U).
\end{equation*}
We denote by $\mathcal{S}^*(\beta)$ the class of  starlike functions of order $\beta$. A function $f\in{\mathcal A}$ is said to be {\it convex of order} $\beta\ (0\leq \beta<1)$, if it satisfies the condition:
\begin{equation*}
\Re\bigg(1+\frac{zf''(z)}{f'(z)}\bigg)>\beta \quad (z\in\U).
\end{equation*}
We denote by $\mathcal{K}(\beta)$ the class of convex functions of order $\beta$.
For simplicity, we also use the notations $\mathcal{S}^{*}:=\mathcal{S}^*(0)$ and $\mathcal{K}:=\mathcal{K}(0)$.

A function $f\in\mathcal{A}$ is said to be {\it strongly starlike} of order $\gamma\ (0\leq \gamma<1)$ if
\begin{equation*}
\left|\arg\left(\frac{zf'(z)}{f(z)}\right)\right|\leq \frac{\pi}{2}\gamma \quad (z\in\U).
\end{equation*}
We denote by $\mathcal{SS}(\gamma)$ the class of strongly starlike functions of order $\gamma$. We also consider the subclass $\mathcal{PS}\subset\mathcal{A}$ of {\it parabolic starlike functions} in $\U$ (see \cite{ron1993}), which satisfy the inequality:
\begin{equation*}
\left|\frac{zf'(z)}{f(z)}-1\right|\leq\Re\left(\frac{zf'(z)}{f(z)}\right) \quad (z\in\U).
\end{equation*}

Recall that an analytic function $w$ of the unit disk $\U$ is a {\it Schwarz function} if it satisfies the conditions
of the Schwarz lemma, i.e.,
$$w(0)=0\ {\rm and}\ \abs{w(z)}<1 \ (z \in {\U}).$$

For two functions $f$ and $g$, analytic in ${\U}$, we say that the function $f$ is {\it subordinate} to $g$ in ${\U}$, and write
$$f(z)\prec g(z)\quad(z\in {\U}),$$
if there exists a Schwarz function $w(z)$, such that
$$f(z)=g\big(w(z)\big)\quad (z\in{\U}).$$

It is well known that if $f(z)\prec g(z)\; (z\in{\U})$, then $f(0)=g(0)$ and $f({\U})\subset g({\U})$.
Furthermore, if the function $g$ is univalent in ${\U}$, then we have the equivalence:
\begin{equation*}
f(z)\prec g(z)\quad (z\in{\U})\; \Longleftrightarrow \; f(0)=g(0)\ {\rm and}\ f({\U})\subset g({\U}).
\end{equation*}

In 1998, Sok\'{o}{\l} \cite{sok1998} introduced the class $\mathcal{SL}\subset\mathcal{S}^{*}$, which consists of the functions $f\in\mathcal{A}$ such that
\begin{equation*}
\frac{zf'(z)}{f(z)}\prec \sqrt{1+z} \quad (z\in\U).
\end{equation*}

In a recent paper, Kargar {\it et al.} \cite{kar2017} investigated the class $\mathcal{MS}(\alpha)$ defined below,
and obtained several radius results for certain well-known function classes.

\begin{definition}\label{d1}
{\rm A function $f\in\mathcal{A}$ is said to belong to the class $\mathcal{MS}(\alpha)\; (\pi/2\leq\alpha<\pi)$, if it satisfies the following conditions:
\begin{equation}\label{12}
1+\frac{\alpha-\pi}{2\sin{\alpha}}<\Re\left(\frac{zf'(z)}{f(z)}\right)<1+\frac{\alpha}{2\sin{\alpha}}  \quad (z\in\U).
\end{equation}
}
\end{definition}

\begin{remark}
{\rm
We note that the inequalities (see \cite{kar2017})
\begin{equation}\label{k}
1-\frac{\pi}{4}\leq 1+\frac{\alpha-\pi}{2\sin{\alpha}}<\frac{1}{2},  \  1+\frac{\alpha}{2\sin{\alpha}}\geq 1+\frac{\pi}{4}  \quad (\pi/2\leq\alpha<\pi).
\end{equation}
It is clear that
\begin{equation*}
\mathcal{MS}(\alpha)\subset\mathcal{S}^{*} \ (\pi/2\leq\alpha<\pi)
\ {\rm and} \ \mathcal{MS}(\pi/2)\subset\mathcal{S}(1-\pi/4,\; 1+\pi/4),
\end{equation*}
where the class $\mathcal{S}(\beta,\; \gamma)$, $0\leq \beta<1<\gamma$, was considered recently by Kwon {\it et al.} in \cite{kscs}.
}
\end{remark}

This paper is organized as follows. In Section 2, we recall certain preliminary lemmas, which are useful in the study of the above classes of functions.
In Section 3, we consider some basic properties of the class $\mathcal{MS}(\alpha)$, such as integral representation, property of convolution,
sufficient condition and coefficient inequalities.
In Section 4, we consider radius problems and inclusion relations for certain classes of strongly starlike functions, parabolic starlike functions
and $\mathcal{SL}\subset\mathcal{S}^{*}$, which are closely related to the class $\mathcal{MS}(\alpha)$,
and the derivations are similar to those used earlier by Sun {\it et al.} \cite{sjr2017} and Kwon {\it et al.} \cite{kscs}.
Our results are essential improvements of the corresponding results obtained by Kargar {\it et al.} \cite{kar2017}.
\vskip.20in

\section{Preliminaries}

Recently, Kargar {\it et al.} \cite{kar2017} introduced the analytic function $F_{\alpha}$ and the vertical strip $\Omega_{\alpha}$,
which are defined as follows:
\begin{equation}\label{13}
F_{\alpha}(z):=\frac{1}{2i\sin{\alpha}}\log\bigg(\frac{1+e^{i\alpha}z}{1+e^{-i\alpha}z}\bigg)  \quad (z\in\U)
\end{equation}
and
\begin{equation}\label{14}
\Omega_{\alpha}:=\left\{\omega\in\mathbb{C}: \; \frac{\alpha-\pi}{2\sin{\alpha}}<\Re\left(\omega\right)<\frac{\alpha}{2\sin{\alpha}}\right\},
\end{equation}
where $\pi/2\leq\alpha<\pi$. The function $F_{\alpha}$, defined by \eqref{13}, is convex and univalent in $\U$.
In addition, $F_{\alpha}$ maps $\U$ onto $\Omega_{\alpha}$, or onto the convex hull of three points
(one of which may be that point at infinity) on the boundary of $\Omega_{\alpha}$.
In other words, the image of $\U$ may be a vertical strip for $\pi/2\leq\alpha<\pi$.
In other cases, the image can be for example a half strip, a quadrilateral, or a triangle (see \cite{dor2001}).

\vskip.05in

We note that the function $F_{\alpha}$ can be written in the form
\begin{equation*}\label{15}
F_{\alpha}(z)=z+\sum_{n=2}^{\infty}B_{n}(\alpha)z^{n}  \quad (\pi/2\leq\alpha<\pi;\; z\in\U),
\end{equation*}
where
\begin{equation}\label{16}
B_{n}(\alpha)=(-1)^{(n-1)}\frac{\sin n\alpha}{n\sin\alpha} \quad (\pi/2\leq\alpha<\pi;\; n\in\N).
\end{equation}

\vskip.05in

In the recent years, there has been significant interesting results about the class of normalized analytic
functions $f\in\mathcal{A}$ that map $\U$ onto vertical strip, see e.g. \cite{kscs, ko, ltmn, sk, sk2016, sjr2017, wsj2015}.

\vskip.05in

In order to prove the main results, we need the following lemmas.

\begin{lemma}\label{lem1}{\rm (see \cite{kar2017})}
Let $f\in\mathcal{A}$. Then $f\in\mathcal{MS}(\alpha)\; (\pi/2\leq\alpha<\pi)$ if and only if
\begin{equation}\label{18}
\bigg(\frac{zf'(z)}{f(z)}-1\bigg)\prec F_{\alpha}(z)=\frac{1}{2i\sin{\alpha}}\log\bigg(\frac{1+e^{i\alpha}z}{1+e^{-i\alpha}z}\bigg)  \quad (z\in\U).
\end{equation}
\end{lemma}

\vskip.05in

\begin{lemma}\label{lem2}{\rm (see \cite{mm1978})}
Let $h$ be analytic and convex univalent in $\U$, and $\beta,\gamma\in\mathbb{R}$ with $\Re\big(\beta h(z)+\gamma\big)\geq 0$.
If $q$ is analytic in $\U$, with $q(0)=h(0)$, then
\begin{equation*}
q(z)+\frac{zq'(z)}{\beta q(z)+\gamma}\prec h(z)\Longrightarrow q(z)\prec h(z)\quad (z\in\U).
\end{equation*}
\end{lemma}

\vskip.05in

\begin{lemma}\label{lem3}{\rm (see \cite{r})}
Let the function $r(z)$ given by
$$r(z)=\sum_{n=1}^{\infty}C_{n}z^{n}$$
be analytic and univalent in $\U$, and suppose that $r(z)$ maps $\U$ onto a convex domain. If the function $q(z)$ given by
$$q(z)=\sum_{n=1}^{\infty}A_{n}z^{n}$$
is analytic in $\U$ and satisfies the following subordination relation:
$$q(z)\prec r(z) \quad (z\in\U),$$
then
$$\abs{A_n}\leq \abs{C_1} \quad (n\in\mathbb{N}).$$
\end{lemma}

\vskip.20in

\section{Properties of the class $\mathcal{MS}(\alpha)$}

In this section, we will study the properties of the class $\mathcal{MS}(\alpha)$.
We begin by giving an integral representation for this class.

\vskip.05in

\begin{theorem}\label{th61}
A function $f\in\mathcal{MS}(\alpha)\; (\pi/2\leq\alpha<\pi)$ if and only if
\begin{equation}\label{61}
f(z)=z\cdot\exp\bigg[\frac{1}{2i\sin{\alpha}}\int_{0}^{z}\frac{1}{t}\log\bigg(\frac{1+e^{i\alpha}w(t)}{1+e^{-i\alpha}w(t)}\bigg)dt\bigg]
\quad \big(z\in\U\big),
\end{equation}
where $w(z)$ is a Schwarz function.
\end{theorem}

\vskip.05in

\begin{proof}
For $f\in\mathcal{MS}(\alpha)$, we know from Lemma \ref{lem1} that \eqref{18} holds. It follows that
\begin{equation}\label{62}
\frac{zf'(z)}{f(z)}-1=\frac{1}{2i\sin{\alpha}}\log\bigg(\frac{1+e^{i\alpha}w(z)}{1+e^{-i\alpha}w(z)}\bigg)  \quad (z\in\U),
\end{equation}
where the Schwarz function $w(z)$ is analytic in $\U$ with
$w(0)=0 \ {\rm and}\  |w(z)|<1 \,(z\in\U).$
We next see from \eqref{62} that
\begin{equation*}
\frac{f'(z)}{f(z)}-\frac{1}{z}=\frac{1}{2iz\sin{\alpha}}\log\bigg(\frac{1+e^{i\alpha}w(z)}{1+e^{-i\alpha}w(z)}\bigg),
\end{equation*}
which, upon integration, yields
\begin{equation}\label{63}
\log\bigg(\frac{f(z)}{z}\bigg)=\frac{1}{2i\sin{\alpha}}\int_{0}^{z}\frac{1}{t}\log\bigg(\frac{1+e^{i\alpha}w(t)}{1+e^{-i\alpha}w(t)}\bigg)dt.
\end{equation}
The assertion \eqref{61} of Theorem \ref{th61} now follows from \eqref{63}.
\end{proof}

\begin{example}
{\rm Let $w(z)=z$ in Theorem \ref{th61}. Then, the function $f_{\alpha}\in\mathcal{MS}(\alpha)\; (\pi/2\leq\alpha<\pi)$ is given by
\begin{equation*}
f_{\alpha}(z)=z\cdot\exp\bigg[\frac{1}{2i\sin{\alpha}}\int_{0}^{z}\frac{1}{t}\log\bigg(\frac{1+e^{i\alpha}t}{1+e^{-i\alpha}t}\bigg)dt\bigg] \quad \big(z\in\U\big).
\end{equation*}
}
\end{example}

Next, we give the following property concerning convolutions for the function class $\mathcal{MS}(\alpha)$.

\begin{theorem}\label{th71}
A function $f\in\mathcal{MS}(\alpha)\; (\pi/2\leq\alpha<\pi)$ if and only if
\begin{equation}\label{71}
f(z)*\bigg\{\frac{z^2}{(1-z)^2}-\frac{z}{1-z}\cdot\frac{1}{2i\sin{\alpha}}\log\bigg(\frac{1+e^{i(\theta+\alpha)}}{1+e^{i(\theta-\alpha)}}\bigg)\bigg\}\neq 0
\quad \big(z\in\U\big),
\end{equation}
where $*$ denotes the Hadamard product, $0<\theta<2\pi$ and $\theta-\alpha\neq\pi$.
\end{theorem}

\begin{proof}
Assume that $f\in\mathcal{MS}(\alpha)$. Then, by Lemma \ref{lem1}, we observe that \eqref{18} holds. This implies that
\begin{equation}\label{72}
\frac{zf'(z)}{f(z)}\neq
1+\frac{1}{2i\sin{\alpha}}\log\bigg(\frac{1+e^{i\alpha}e^{i\theta}}{1+e^{-i\alpha}e^{i\theta}}\bigg)
\quad \big(0<\theta<2\pi,\ \theta-\alpha\neq\pi;\ z\in\U\big).
\end{equation}
The condition \eqref{72} can now be written as follows:
\begin{equation}\label{73}
zf'(z)-\left[1+\frac{1}{2i\sin{\alpha}}\log\bigg(\frac{1+e^{i(\theta+\alpha)}}{1+e^{i(\theta-\alpha)}}\bigg)\right]f(z)\neq 0
\quad \big(0<\theta<2\pi,\ \theta-\alpha\neq\pi;\ z\in\U\big).
\end{equation}
We note that
\begin{equation}\label{74}
f(z)=f(z)*\left(\frac{z}{1-z}\right)  \ {\rm and} \  zf'(z)=f(z)*\left(\frac{z}{(1-z)^2}\right).
\end{equation}
Thus, by virtue of \eqref{73} and \eqref{74}, we obtain the assertion \eqref{71} of Theorem \ref{th71}.
\end{proof}

We now derive a sufficient condition involving subordination for the functions to be in the class $\mathcal{MS}(\alpha)$.

\begin{theorem}\label{thsub1}
Let $f\in\mathcal{A}$ and satisfy the following subordination
\begin{equation}\label{sub1}
\bigg(1+\frac{zf''(z)}{f'(z)}\bigg)\prec 1+F_{\alpha}(z)  \quad (z\in\U).
\end{equation}
Then
\begin{equation}\label{sub2}
\frac{zf'(z)}{f(z)}\prec 1+F_{\alpha}(z)  \quad (z\in\U),
\end{equation}
that is, $f\in\mathcal{MS}(\alpha)$, where $F_{\alpha}$ is given by \eqref{13}.
\end{theorem}

\begin{proof}
We consider the function $p(z)$ such that
\begin{equation}\label{sub3}
p(z)+1=\frac{zf'(z)}{f(z)} \quad (z\in\U).
\end{equation}
Therefore,
\begin{equation*}
\log\big(p(z)+1\big)+\log\frac{f(z)}{z}=\log f'(z).
\end{equation*}
We have
\begin{equation*}
\frac{p'(z)}{p(z)+1}+\frac{f'(z)}{f(z)}-\frac{1}{z}=\frac{f''(z)}{f'(z)},
\end{equation*}
or
\begin{equation*}
\frac{zp'(z)}{p(z)+1}+\frac{zf'(z)}{f(z)}-1=\frac{zf''(z)}{f'(z)}.
\end{equation*}
From \eqref{sub1} and \eqref{sub3}, we have
\begin{equation}\label{sub4}
\frac{zp'(z)}{p(z)+1}+p(z)=\frac{zf''(z)}{f'(z)}\prec F_{\alpha}(z).
\end{equation}
We note that
\begin{equation}\label{sub5}
p(0)=0=F_{\alpha}(0) \ {\rm and }  \  \Re\big(1+F_{\alpha}(z)\big)>0 \quad (\pi/2\leq \alpha<\pi;\ z\in\U).
\end{equation}
Moreover, by \eqref{sub4} and \eqref{sub5} in Lemma \ref{lem2}, we have
\begin{equation*}
p(z)\prec F_{\alpha}(z),
\end{equation*}
or by \eqref{sub3}, we have
\begin{equation*}
\frac{zf'(z)}{f(z)}\prec 1+F_{\alpha}(z)  \quad (z\in\U).
\end{equation*}
Therefore, by Lemma \ref{lem1}, we obtain that $f\in\mathcal{MS}(\alpha)$.
\end{proof}

\begin{remark}
{\rm
It is well known that $\mathcal{K}\subset\mathcal{S}^{*}(1/2)$. In view of \eqref{k} and Theorem \ref{thsub1}, we can obtain
$\mathcal{K}\big(\Phi(\alpha)\big)\subset\mathcal{S}^{*}\big(\Phi(\alpha)\big)$ for $\pi/2\leq\alpha<\pi$, where $\Phi(\alpha)=1+(\alpha-\pi)/(2\sin\alpha)$.
}
\end{remark}

Now, we present the bounds on the coefficients for functions of the class $\mathcal{MS}(\alpha)$.
The basic method of proof is similar to that used in \cite[Theorem 3.1]{sunyong3}.

\begin{theorem}\label{th21}
Let $f(z)=z+\sum_{n=2}^{\infty}a_{n}z^{n}\in\mathcal{MS}(\alpha)$. Then
$$|a_n|\leq 1 \quad (n\in\mathbb{N}).$$
\end{theorem}

\begin{proof}
For given $\alpha\ (\pi/2\leq\alpha<\pi)$, we define the functions $q(z)$ and $p(z)$ by
\begin{equation}\label{24}
q(z)=\frac{zf'(z)}{f(z)} \quad (z\in\U)
\end{equation}
and
\begin{equation}\label{25}
p_{\alpha}(z)=1+\frac{1}{2i\sin{\alpha}}\log\bigg(\frac{1+e^{i\alpha}z}{1+e^{-i\alpha}z}\bigg)   \quad \big(z\in\U\big).
\end{equation}
Then the subordination \eqref{18} can be written as follows:
\begin{equation}\label{26}
q(z)\prec p_{\alpha}(z) \quad (z\in\U).
\end{equation}
We note that the function $p(z)$ defined by \eqref{25} is convex in $\U$ and has the form:
\begin{equation*}\label{27}
p_{\alpha}(z)=1+\sum_{n=1}^{\infty}B_{n}(\alpha)z^{n}  \quad \big(z\in\U\big),
\end{equation*}
where $B_{n}(\alpha)$ is given by \eqref{16}. If we let
\begin{equation*}\label{29}
q(z)=1+\sum_{n=1}^{\infty}A_{n}z^{n}  \quad (z\in\U),
\end{equation*}
then, by Lemma \ref{lem3}, we see that the subordination \eqref{26} implies that
\begin{equation}\label{210}
|A_n|\leq |B_1|=1 \quad (n\in\mathbb{N}).
\end{equation}

Now, \eqref{24} implies that
\begin{equation*}
zf'(z)=f(z)q(z) \quad (z\in\U).
\end{equation*}
Then, by equating the coefficients of $z^n$ on both sides, we get
\begin{equation*}
a_n=\frac{1}{n-1}\big(A_{n-1}+a_{2}A_{n-2}+a_{3}A_{n-3}+\cdots+a_{n-1}A_{1}\big) \quad \big(n\in\mathbb{N}\setminus\{1\}\big).
\end{equation*}
A simple calculation combined with the inequality \eqref{210} yields $|a_2|=|A_1|\leq 1$ and
\begin{equation*}\label{2011}
\begin{split}
|a_n|&=\frac{1}{n-1}\big|A_{n-1}+a_{2}A_{n-2}+a_{3}A_{n-3}+\cdots+a_{n-1}A_{1}\big|\\
&\leq\frac{1}{n-1}\big(|A_{n-1}|+|a_{2}|\cdot|A_{n-2}|+|a_{3}|\cdot|A_{n-3}|+\cdots+|a_{n-1}|\cdot|A_{1}|\big)\\
&\leq\frac{|B_1|}{n-1}\bigg(1+\sum_{k=2}^{n-1}|a_{k}|\bigg)=\frac{1}{n-1}\bigg(1+\sum_{k=2}^{n-1}|a_{k}|\bigg) \quad \big(n\in\mathbb{N}\setminus\{1,2\}\big).
\end{split}
\end{equation*}
To prove the assertion of Theorem \ref{th21}, we need to show that
\begin{equation}\label{211}
|a_n|\leq \frac{1}{n-1}\bigg(1+\sum_{k=2}^{n-1}|a_{k}|\bigg)\leq 1  \quad \big(n\in\mathbb{N}\setminus\{1,2\}\big).
\end{equation}
We now use mathematical induction to prove \eqref{211}. For $n=3$, we have
\begin{equation*}
|a_3|\leq \frac{1}{2}\big(1+|a_{2}|\big)\leq \frac{1}{2}\big(1+|A_{1}|\big)\leq 1.
\end{equation*}
Then suppose that the inequality \eqref{211} is true for $3\leq n\leq m$. We prove the statement for $n=m+1$. Straightforward calculations yield
\begin{equation*}
\begin{split}
|a_{m+1}|
&\leq \frac{1}{m}\bigg(1+\sum_{k=2}^{m}|a_{k}|\bigg)\leq \frac{1}{m}\big[1+(m-1)\big]=1,
\end{split}
\end{equation*}
which implies that the inequality \eqref{211} is true for $n=m+1$.
\end{proof}

\section{Radius problems and inclusion relations}

In this section, we first give results on the radius problem involving the function class $\mathcal{MS}(\alpha)$.
As an application, we obtain inclusion relations for the class $\mathcal{MS}(\alpha)$ and the other well-known function classes.
The basic method of proof in the following theorem is similar to that used in \cite[Theorem 5]{sjr2017} (see also \cite[Theorem 3.1]{kscs}).
\vskip.05in

\begin{theorem}\label{th81}
Let $f\in\mathcal{MS}(\alpha)$. Then, for each $z\
(|z|=r<1)$,
\begin{equation}\label{81}
1+\frac{1}{2\sin{\alpha}}\big[M_{1}(r,\alpha)-M_{2}(r,\alpha)\big]
\leq \Re\left(\frac{zf'(z)}{f(z)}\right) \leq
1+\frac{1}{2\sin{\alpha}}\big[M_{1}(r,\alpha)+M_{2}(r,\alpha)\big]
\end{equation}
and
\begin{equation}\label{82}
\bigg|\Im\left(\frac{zf'(z)}{f(z)}\right)\bigg| \leq \frac{1}{2\sin{\alpha}}\log\big[N(r,\alpha)\big],
\end{equation}
where
\begin{equation}\label{83}
M_{1}(r,\alpha)=\arcsin\bigg(\frac{-r^{2}\sin{2\alpha}}{\sqrt{1-2r^{2}\cos{2\alpha}+r^{4}}}\bigg),
\end{equation}

\begin{equation}\label{84}
M_{2}(r,\alpha)=\arcsin\bigg(\frac{2r\sin{\alpha}}{\sqrt{1-2r^{2}\cos{2\alpha}+r^{4}}}\bigg),
\end{equation}

\begin{equation}\label{85}
N(r,\alpha)=\frac{\sqrt{1-2r^{2}\cos{2\alpha}+r^{4}}+2r\sin\alpha}{1-r^{2}}.
\end{equation}
\end{theorem}

\begin{proof}
Suppose that $f\in\mathcal{MS}(\alpha)$. Then, by Lemma \ref{lem1}, the assertion \eqref{18} holds.
Thus, by the definition of subordination, there exists a Schwarz function
$w(z)$ such that
\begin{equation*}
\frac{zf'(z)}{f(z)}
=1+\frac{1}{2i\sin{\alpha}}\log\bigg(\frac{1+e^{i\alpha}w(z)}{1+e^{-i\alpha}w(z)}\bigg)
\quad (z\in\U).
\end{equation*}
We put
\begin{equation*}
Q(z)=\frac{1+e^{i\alpha}w(z)}{1+e^{-i\alpha}w(z)} \quad (z\in\U),
\end{equation*}
which readily yields
\begin{equation*}
Q(z)-1=\big[e^{i\alpha}-e^{-i\alpha}Q(z)\big]w(z).
\end{equation*}
For $|z|=r<1$, using the Schwarz lemma,
$$\abs{w(z)}\leq \abs{z} \quad (z\in\U),$$
we find that
\begin{equation}\label{88}
|Q(z)-1|\leq \big|e^{i\alpha}-e^{-i\alpha}Q(z)\big|r \quad (|z|=r<1).
\end{equation}
If we set $Q(z)=u+iv$, then, upon squaring both sides of \eqref{88}, we get
\begin{equation}\label{89}
\bigg(u-\frac{1-r^2\cos2\alpha}{1-r^2}\bigg)^2+\bigg(v+\frac{r^2\sin2\alpha}{1-r^2}\bigg)^2
\leq \bigg(\frac{2r\sin\alpha}{1-r^2}\bigg)^2.
\end{equation}
Thus, $Q(z)$ maps the disk
$$\overline{\U_{r}}=\{z: z\in\mathbb{C} \ \text{and} \ |z|\leq r<1\}$$
onto the disk which the center $C$ and radius $R$ are given by
\begin{equation}\label{circle89}
C:=\left(\frac{1-r^2\cos2\alpha}{1-r^2},-\frac{r^2\sin2\alpha}{1-r^2}\right),
\quad
R:=\frac{2r\sin\alpha}{1-r^2},
\end{equation}
respectively.

We observe that
\begin{equation*}
X_{C}:=\frac{1-r^2\cos2\alpha}{1-r^2}>0, \ Y_{C}:=-\frac{r^2\sin2\alpha}{1-r^2}>0, \ R=\frac{2r\sin\alpha}{1-r^2}>0  \quad (\pi/2\leq \alpha<\pi)
\end{equation*}
and
\begin{equation*}
\big|\overrightarrow{OC}\big|^{2}-R^{2}=1>0, \ Y_{C}-R=\frac{2r\sin\alpha(1-\cos\alpha)}{1-r^2}<0  \quad (\pi/2\leq \alpha<\pi).
\end{equation*}
Hence, the origin $O$ lies outside of the disk \eqref{circle89}, and the disk \eqref{circle89} lies in the first and the forth quadrants of $uv$-plane.

\vskip.05in

We can obtain the upper and the lower bounds of $|Q(z)|$:

\begin{equation}\label{810}
|Q(z)|\leq \big|\overrightarrow{OC}\big|+R=\frac{\sqrt{1-2r^{2}\cos{2\alpha}+r^{4}}+2r\sin\alpha}{1-r^{2}}=:N(r,\alpha)
\end{equation}
and
\begin{equation}\label{811}
|Q(z)|\geq\big|\overrightarrow{OC}\big|-R=\frac{\sqrt{1-2r^{2}\cos{2\alpha}+r^{4}}-2r\sin\alpha}{1-r^{2}}=\frac{1}{N(r,\alpha)},
\end{equation}
where $N(r,\alpha)>1$ is already given by \eqref{85}.

Furthermore, a simple geometric observation shows that \eqref{89} implies
\begin{equation}\label{812}
M_{1}(r,\alpha)-M_{2}(r,\alpha)  \leq \arg\big(Q(z)\big)\leq  M_{1}(r,\alpha)+M_{2}(r,\alpha),
\end{equation}
where $M_{1}(r,\alpha)$ and $M_{2}(r,\alpha)$ are given by \eqref{83} and \eqref{84}, respectively.

For $|z|=r<1$, we have
\begin{equation}\label{813}
\begin{split}
\frac{zf'(z)}{f(z)}
&=1+\frac{1}{2i\sin{\alpha}}\log\big(Q(z)\big)\\
&=1+\frac{1}{2i\sin{\alpha}}\big[\log\big|Q(z)\big|+i\arg\big(Q(z)\big)\big]\\
&=1+\frac{1}{2\sin{\alpha}}\arg\big(Q(z)\big)-\frac{i}{2\sin{\alpha}}\log\big|Q(z)\big|.
\end{split}
\end{equation}
Thus, by virtue of \eqref{810}-\eqref{813}, we easily get the assertions \eqref{81} and \eqref{82} of Theorem \ref{th81}.
\end{proof}

The following identities are used in the proofs of our main results:
\begin{equation}\label{r0}
\lim_{r\rightarrow 0+}M_{1}(r,\alpha)=\lim_{r\rightarrow 0+}M_{2}(r,\alpha)=0, \quad \lim_{r\rightarrow 0+}N(r,\alpha)=1,
\end{equation}
and
\begin{equation}\label{r1}
\lim_{r\rightarrow 1-}M_{1}(r,\alpha)=\frac{3\pi}{2}-\alpha, \ \lim_{r\rightarrow 1-}M_{2}(r,\alpha)=\frac{\pi}{2}, \ \lim_{r\rightarrow 1-}N(r,\alpha)=+\infty,
\end{equation}
where $M_{1}(r,\alpha)$, $M_{2}(r,\alpha)$ and $N(r,\alpha)$ are given by \eqref{83}, \eqref{84} and \eqref{85}, respectively.

By using Theorem \ref{th81}, we derive the following inclusion relations for the class $\mathcal{MS}(\alpha)$.

\vskip.05in

\begin{theorem}\label{th31}
Assume that
\begin{equation*}
\frac{\pi}{2}\leq\alpha<\pi \ {\rm and} \ 0\leq\gamma<1.
\end{equation*}
Then
\begin{equation*}
\mathcal{MS}(\alpha)\subset\mathcal{SS}(\gamma) \quad (|z|\leq r_{1}),
\end{equation*}
where $r_{1}\in(0,1)$ is the least positive root of the following equation:
\begin{equation*}
\arctan\bigg(\frac{\log\big[N(r,\alpha)\big]}{2\sin{\alpha}+M_{1}(r,\alpha)-M_{2}(r,\alpha)}\bigg)
-\frac{\pi}{2}\gamma=0 \quad \big(0\leq r<1\big),
\end{equation*}
where $M_{1}(r,\alpha)$, $M_{2}(r,\alpha)$ and $N(r,\alpha)$ are given by \eqref{83}, \eqref{84} and \eqref{85}, respectively.
\end{theorem}

\begin{proof}
We first note that
\begin{equation*}
1+\frac{1}{2\sin{\alpha}}\big[M_{1}(r,\alpha)-M_{2}(r,\alpha)\big]>0
\quad \big(\pi/2\leq\alpha<\pi;\; 0\leq r<1\big).
\end{equation*}
Hence, by Theorem \ref{th81}, for $f\in\mathcal{MS}(\alpha)$, we have
\begin{equation*}
\left|\arg\left(\frac{zf'(z)}{f(z)}\right)\right|
\leq\arctan\bigg(\frac{\frac{1}{2\sin{\alpha}}\log\big[N(r,\alpha)\big]}
{1+\frac{1}{2\sin{\alpha}}\big[M_{1}(r,\alpha)-M_{2}(r,\alpha)\big]}\bigg).
\end{equation*}
Thus, for the function $f\in\mathcal{SS}(\gamma)$, it suffices to prove the following inequality:
\begin{equation*}
\arctan\bigg(\frac{\log\big[N(r,\alpha)\big]}{2\sin{\alpha}+M_{1}(r,\alpha)-M_{2}(r,\alpha)}\bigg)
-\frac{\pi}{2}\gamma<0.
\end{equation*}
We now define a continuous function $G(r)$ by
\begin{equation*}
G(r)=\arctan\bigg(\frac{\log\big[N(r,\alpha)\big]}{2\sin{\alpha}+M_{1}(r,\alpha)-M_{2}(r,\alpha)}\bigg)
-\frac{\pi}{2}\gamma  \quad  (0\leq r<1).
\end{equation*}
In view of \eqref{r0} and \eqref{r1}, we can show that
\begin{equation*}
G(0)=-\frac{\pi}{2}\gamma<0 \ {\rm and} \ \lim_{r\rightarrow1-}G(r)=\frac{\pi}{2}-\frac{\pi}{2}\gamma>0.
\end{equation*}
Thus, the equation $G(r)=0$ has a solution in $(0,1)$. Let $r_1\in(0,1)$ be the least positive root of $G(r)=0$.
Then $G(r)<0$ for all $r<r_{1}$. Hence, $f$ is a strongly starlike function of order $\gamma$ for $z\ (|z|\leq r_{1})$.
\end{proof}

\begin{theorem}\label{th41}
Let $\pi/2\leq\alpha<\pi$. Then
\begin{equation*}
\mathcal{MS}(\alpha)\subset\mathcal{PS} \quad (|z|\leq r_{2}),
\end{equation*}
where $r_{2}\in(0,1)$ is the least positive root of the equation:
\begin{equation*}
\frac{1}{4\sin^{2}{\alpha}}\bigg\{\log\big[N(r,\alpha)\big]\bigg\}^{2}-\frac{1}{\sin{\alpha}}\big[M_{1}(r,\alpha)-M_{2}(r,\alpha)\big]-1=0  \quad  \big(0\leq r<1\big),
\end{equation*}
where $M_{1}(r,\alpha)$, $M_{2}(r,\alpha)$ and $N(r,\alpha)$ are given by \eqref{83}, \eqref{84} and \eqref{85}, respectively.
\end{theorem}

\begin{proof}
We note that $f\in\mathcal{PS}$ if and only if the function $zf'(z)/f(z)$ is in the parabolic region given by
\begin{equation*}
\Lambda=\{(u,v):\ v^2<2u-1\}.
\end{equation*}
Thus, by combining \eqref{81} and \eqref{82}, for the function $f\in\mathcal{PS}$ in $\U$, it suffices to show that
\begin{equation*}
\bigg(1+\frac{1}{2\sin{\alpha}}\big[M_{1}(r,\alpha)-M_{2}(r,\alpha)\big],\; \frac{1}{2\sin{\alpha}}\log\big[N(r,\alpha)\big]\bigg)\in\Lambda,
\end{equation*}
that is,
\begin{equation*}
\frac{1}{4\sin^{2}{\alpha}}\bigg\{\log\big[N(r,\alpha)\big]\bigg\}^{2}-\frac{1}{\sin{\alpha}}\big[M_{1}(r,\alpha)-M_{2}(r,\alpha)\big]-1<0.
\end{equation*}
We now define a continuous function $H(r)$ by
\begin{equation*}
H(r):=\frac{1}{4\sin^{2}{\alpha}}\bigg\{\log\big[N(r,\alpha)\big]\bigg\}^{2}-\frac{1}{\sin{\alpha}}\big[M_{1}(r,\alpha)-M_{2}(r,\alpha)\big]-1  \quad (0\leq r<1).
\end{equation*}
In view of \eqref{r0} and \eqref{r1}, we have
\begin{equation*}
H(0)=-1<0 \ {\rm and} \ \lim_{r\rightarrow1-}H(r)=+\infty.
\end{equation*}
Hence the equation $H(r)=0$ has a solution in $(0,1)$. Let $r_{2}\in(0,1)$ be the least positive root of $H(r)=0$.
Then $H(r)<0$ for all $r<r_{2}$. Therefore, we have $f\in\mathcal{PS}$ for all $z\ (|z|\leq r_{2})$.
\end{proof}

\begin{theorem}\label{th51}
Let $\pi/2\leq\alpha<\pi$. Then
\begin{equation*}
\mathcal{MS}(\alpha)\subset\mathcal{SL} \qquad (|z|\leq r_{0}),
\end{equation*}
where $r_{0}:=\min\{r_{3},\; r_{4}\}$, $r_{3},\; r_{4}\in(0,1) $ are the least positive root of the equations:
\begin{equation}\label{r3}
\begin{split}
&\ \ \ \ \  \bigg\{\bigg(1+\frac{1}{2\sin{\alpha}}\big[M_{1}(r,\alpha)+M_{2}(r,\alpha)\big]\bigg)^{2}
+\frac{1}{4\sin^{2}{\alpha}}\bigg(\log\big[N(r,\alpha)\big]\bigg)^{2}\bigg\}^{2}\\
&=2\bigg(1+\frac{1}{2\sin{\alpha}}\big[M_{1}(r,\alpha)+M_{2}(r,\alpha)\big]\bigg)^{2}
-\frac{1}{2\sin^{2}{\alpha}}\bigg(\log\big[N(r,\alpha)\big]\bigg)^{2},
\end{split}
\end{equation}
and
\begin{equation}\label{r4}
\begin{split}
&\ \ \ \ \  \bigg\{\bigg(1+\frac{1}{2\sin{\alpha}}\big[M_{1}(r,\alpha)-M_{2}(r,\alpha)\big]\bigg]^{2}
+\frac{1}{4\sin^{2}{\alpha}}\bigg(\log\big[N(r,\alpha)\big]\bigg)^{2}\bigg\}^{2}\\
&=2\bigg(1+\frac{1}{2\sin{\alpha}}\big[M_{1}(r,\alpha)-M_{2}(r,\alpha)\big]\bigg)^{2}
-\frac{1}{2\sin^{2}{\alpha}}\bigg(\log\big[N(r,\alpha)\big]\bigg)^{2},
\end{split}
\end{equation}
respectively, where $M_{1}(r,\alpha)$, $M_{2}(r,\alpha)$ and $N(r,\alpha)$ are given by \eqref{83}, \eqref{84} and \eqref{85}, respectively.
\end{theorem}

\begin{proof}
We note that $f\in\mathcal{SL}$ if and only if the function $zf'(z)/f(z)$ is in the bounded region given by
\begin{equation*}
\Xi=\big\{(u,v):\ \big(u^{2}+v^{2}\big)^{2}<2\big(u^{2}-v^{2}\big)\big\}.
\end{equation*}
Thus, by combining \eqref{81} and \eqref{82}, for the function $f\in\mathcal{SL}$ in $\U$, it suffices to show that
\begin{equation*}
\bigg(1+\frac{1}{2\sin{\alpha}}\big[M_{1}(r,\alpha)\pm M_{2}(r,\alpha)\big],\; \frac{1}{2\sin{\alpha}}\log\big[N(r,\alpha)\big]\bigg)\in\Xi,
\end{equation*}
that is,
\begin{equation}\label{p3}
\begin{split}
&\ \ \ \ \  \bigg\{\bigg(1+\frac{1}{2\sin{\alpha}}\big[M_{1}(r,\alpha)+M_{2}(r,\alpha)\big]\bigg)^{2}
+\frac{1}{4\sin^{2}{\alpha}}\bigg(\log\big[N(r,\alpha)\big]\bigg)^{2}\bigg\}^{2}\\
&<2\bigg(1+\frac{1}{2\sin{\alpha}}\big[M_{1}(r,\alpha)+M_{2}(r,\alpha)\big]\bigg)^{2}
-\frac{1}{2\sin^{2}{\alpha}}\bigg(\log\big[N(r,\alpha)\big]\bigg)^{2},
\end{split}
\end{equation}
and
\begin{equation}\label{p4}
\begin{split}
&\ \ \ \ \  \bigg\{\bigg(1+\frac{1}{2\sin{\alpha}}\big[M_{1}(r,\alpha)-M_{2}(r,\alpha)\big]\bigg)^{2}
+\frac{1}{4\sin^{2}{\alpha}}\bigg(\log\big[N(r,\alpha)\big]\bigg)^{2}\bigg\}^{2}\\
&<2\bigg(1+\frac{1}{2\sin{\alpha}}\big[M_{1}(r,\alpha)-M_{2}(r,\alpha)\big]\bigg)^{2}
-\frac{1}{2\sin^{2}{\alpha}}\bigg(\log\big[N(r,\alpha)\big]\bigg)^{2},
\end{split}
\end{equation}
respectively. We define a continuous function $P(r)$ by
\begin{equation*}
\begin{split}
P(r):=&\bigg\{\bigg(1+\frac{1}{2\sin{\alpha}}\big[M_{1}(r,\alpha)+M_{2}(r,\alpha)\big]\bigg)^{2}
+\frac{1}{4\sin^{2}{\alpha}}\bigg(\log\big[N(r,\alpha)\big]\bigg)^{2}\bigg\}^{2}\\
&-2\bigg(1+\frac{1}{2\sin{\alpha}}\big[M_{1}(r,\alpha)+M_{2}(r,\alpha)\big]\bigg)^{2}
+\frac{1}{2\sin^{2}{\alpha}}\bigg(\log\big[N(r,\alpha)\big]\bigg)^{2} \quad (0\leq r<1)
\end{split}
\end{equation*}
In view of \eqref{r0} and \eqref{r1}, we have
\begin{equation*}
P(0)=-1<0 \ {\rm and} \ \lim_{r\rightarrow1-}P(r)=+\infty.
\end{equation*}
Hence the equation $P(r)=0$ has a solution in $(0,1)$. Let $r_{3}\in(0,1)$ be the least positive root of $H(r)=0$.
Then $P(r)<0$ for all $r<r_{3}$. Using the same approach as above, we can find $r_{4}\in(0,1)$ be the least positive root of the equation \eqref{r4},
and the inequality \eqref{p4} holds for all $r<r_{4}$. So if we take $r_{0}:=\min\{r_{3},\; r_{4}\}$, then we have $f\in\mathcal{SL}$ for all $z\ (|z|\leq r_{0})$.
\end{proof}

\begin{remark}
{\rm
Putting $\alpha=\pi/2$ in Theorems \ref{th31}-\ref{th51}, we obtain the radii of inclusion relations between several known classes and the class $\mathcal{MS}(\alpha)$.
Furthermore, the results are compared with the corresponding results in \cite{kar2017} (see Table 1).

\begin{table}[!hbp]
\begin{center}
{
\caption{The radii of inclusion relations}
\vskip.05in
\begin{tabular}{|c|c|c|}
\hline
\hline
Inclusion relations  & Radii in this paper & Radii in Ref. \cite{kar2017} \\
\hline
$\mathcal{MS}(\pi/2)\subset \mathcal{SS}(1/2) \quad (|z|\leq r_{1})$        & $r_{1}\approx 0.493918$   & $r_{1}\approx 0.260446$ \\
\hline
$\mathcal{MS}(\pi/2)\subset \mathcal{PS}  \quad (|z|\leq r_{2})$            & $r_{2}\approx 0.421547$   & $r_{2}\approx 0.246969$ \\
\hline
$\mathcal{MS}(\pi/2)\subset \mathcal{SL}  \quad (|z|\leq r_{0})$            & $r_{0}\approx 0.304506$   & $r_{0}\approx 0.200667$ \\
\hline
\end{tabular}
}
\end{center}
\end{table}
}
\end{remark}

\vskip .20in
\begin{center}{\sc Acknowledgments}
\end{center}

\vskip .05in
The present investigation was supported by the \textit{Natural Science Foundation of Hunan Province} under Grant no. 2016JJ2036 of the People's Republic of China.

\end{document}